\documentclass{amsart}
\usepackage{graphicx}
\usepackage{verbatim}
\usepackage{cite}
\usepackage{esint} 
\newtheorem{thm}{Theorem}[section]
\newtheorem{cor}[thm]{Corollary}
\newtheorem{lem}[thm]{Lemma}

\theoremstyle{definition}
\newtheorem{defn}[thm]{Definition}
\theoremstyle{remark}
\newtheorem{rem}[thm]{Remark}
\newtheorem{claim}[thm]{Claim}
\numberwithin{equation}{section}

\newcommand{\Real}{\mathbb R}

\title[Regularity results for biharmonic maps]{Quantitative stratification and higher regularity for biharmonic maps}
\author[C.~Breiner]{Christine~Breiner}

\address[C.~Breiner]{Department of Mathematics, Fordham University, Bronx, NY 10458}
\email{cbreiner@fordham.edu}

\author[T.~Lamm]{Tobias~Lamm}
\address[T.~Lamm]{Institute for Analysis, Karlsruhe Institute of Technology (KIT), Kaiserstr. 89-93, D-76133 Karlsruhe, Germany}
\email{tobias.lamm@kit.edu}

\thanks{The first author was supported in part by NSF grant DMS-1308420.}


\begin{document}

\maketitle
\begin{abstract}
In this paper we prove quantitative regularity results for stationary and minimizing extrinsic biharmonic maps. As an application, we determine sharp, dimension independent $L^p$ bounds for $\nabla^k f$ that do not require a small energy hypothesis. In particular, every minimizing biharmonic map is in $W^{4,p}$ for all $1\le p<5/4$. Further, for minimizing biharmonic maps from $\Omega \subset \Real^5$, we determine a uniform bound on the number of singular points in a compact set. Finally, using dimension reduction arguments, we extend these results to minimizing and stationary biharmonic maps into special targets.
\end{abstract}
\section{Introduction}
In this paper we refine the regularity theory for minimizing and stationary biharmonic maps $f\in W^{2,2}(\Omega,N^n)$ where $\Omega \subset \Real^m$, $m \geq 4$, $N$ is a compact, smooth manifold and $N\subset \Real^\ell$ has no boundary. Known regularity results for harmonic maps have biharmonic analogues and many of the proofs in the biharmonic setting are inspired by the proofs for harmonic maps. 
Our regularity results rely on the technique known as \emph{quantitative stratification}, first developed by Cheeger and Naber in \cite{CNInvent} and later used to refine the regularity theory for harmonic maps and minimal currents \cite{CNCPAM}.

The first regularity results for biharmonic maps are due to Chang, L. Wang, and Yang \cite{ChangWangYang} who consider $N = \mathbb S^{\ell-1}$. In concert with harmonic results \cite{Helein,Evans2,ChangWangYang2}, they determine that all biharmonic maps into spheres are smooth for $m=4$ and for stationary biharmonic $f:\Omega^m \to \mathbb S^{\ell-1}$ with $m\geq 5$, $\dim\mathcal S(f) \leq m-4$. Here we define $\mathcal S(f):=\Omega\backslash \{ x \in \Omega: f\text{ is } C^\infty \text{ in a neighborhood of }x\}$ and $\dim X$ represents the Hausdorff dimension of a set $X$. Somewhat later, Strzelecki and C. Wang separately provided alternate proofs \cite{Wang,Strzelecki}. Following the ideas of H\'elein \cite{Helein}, each of the proofs first exploit the symmetry of the target manifold. Chang, L. Wang, and Yang write the non-linear part of the Euler-Lagrange equation in divergence form while Strzelecki and C. Wang each observe that the full Euler-Lagrange equation is equivalent to a particular wedge product being divergence free. The biharmonic proofs then deviate from the classical Hardy space methods of \cite{Helein} as the structure of the higher order problem makes it difficult to verify that the appropriate lower order terms are in $\mathcal H^1$. Instead, \cite{ChangWangYang} uses singular integral estimates to determine an appropriate decay estimate and H\"older continuity, \cite{Strzelecki} replaces the Hardy methods with a weaker duality lemma that still allows one to conclude a reverse H\"older inequality, and \cite{Wang} uses a Coulomb gauge frame and estimates on Riesz potentials in Morrey spaces. 

C. Wang extends his result on biharmonic maps to general targets in \cite{WangMathZ,Wang2}. Again using a Coulomb gauge frame and Riesz potentials or Lorentz space estimates, he establishes the necessary decay estimates. He shows that every biharmonic $f:B^4 \to N$ is smooth and every stationary biharmonic $f:B^m \to N$ with $m \geq 5$ satisfies $\dim\mathcal S(f) \leq m-4$. The second author and Rivi\`ere \cite{LammRiv} and  Struwe \cite{Struwe} later provided an alternate proof of the (partial) regularity, extending the lower order gauge theory technique developed in \cite{Riv} and \cite{RiviereStruwe}. See \cite{Helein2,Bethuel} for analogous harmonic results. Finally, Scheven \cite{Scheven} uses the analysis of defect measures to improve the codimension bound on the singular set for minimizers to $m-5$ and lowers the dimension further for special targets. Scheven's proof follows ideas developed by Lin \cite{Lin} for minimizing and stationary harmonic maps. See also \cite{SchoenUhlenbeck,Luckhaus} for similar results on harmonic maps. 

Let $f \in W^{2,2}(\Omega,N)$ be a stationary biharmonic map and let $\mathcal S^k(f)$ denote the set of points in $\Omega$ where $y \in \mathcal S^k(f)$ if no tangent map of $f$ at $y$ is $k+1$-homogeneous. (See Definition \ref{homdef} and following for the precise definitions.) Then we have the natural stratification
\[
\mathcal S^0(f) \subset \mathcal S^1(f) \subset \cdots \subset \mathcal S^{m-1}(f) = \mathcal S(f) \subset \Omega.
\]Moreover, by \cite{White_Strat}
\begin{equation}\label{SingDim}
\dim\mathcal S^k(f) \leq k 
\end{equation}
and by \cite{Scheven}, if $f$ is minimizing biharmonic then
\begin{equation}\label{SingCod}
 \mathcal S^{m-5}(f)= \mathcal S(f).
\end{equation}




Our first result is a quantitative refinement of \eqref{SingDim} for stationary biharmonic maps. In Theorem \ref{minkowskibound}, we demonstrate an effective refinement of the Hausdorff dimension bound on \emph{quantitative singular strata} $\mathcal S^k_{\eta,r}(f)$, see Definition \ref{qssdef}. This result also improves the Hausdorff dimension bound on the singular set, described in \cite{Scheven}, to a Minkowski bound. Our second main result gives the higher regularity for minimizers. In Theorem \ref{epsreg}, we demonstrate that a sufficiently high degree of almost homogeneity for minimizers implies normalized $C^4$ bounds on a fixed scale.  
We then use the volume bound of Theorem \ref{minkowskibound} to get precise control on the size of the set of points where this almost homogeneity fails on fixed scales. Controlling the size of the bad set allows us to determine $L^p$ bounds for the reciprocal of the regularity scale function, see Theorem \ref{MinMinkowski}. These $L^p$ bounds immediately imply dimension independent, sharp $L^p$ estimates on $\nabla f, \nabla^2 f, \nabla^3 f, \nabla^4 f$. These regularity results are the first of their kind and are stated in Corollary \ref{MinCor}. Finally in Theorem \ref{counting}, for minimizing and stationary biharmonic maps where domain dimension and target hypotheses imply that $\dim \mathcal S(f)=0$, we use the techniques developed in \cite{NVV} to determine a uniform upper bound on the number of singular points in a compact set.

We remark that the results of this paper also hold for the intrinsically defined second-order functional considered first by Moser in \cite{Moser} and later by Scheven in \cite{Scheven2}. Indeed, all of the techniques of this paper immediately extend to the intrinsic setting and the necessary compactness theory and dimension bounds for minimizers to the intrinsic energy are proven in \cite{Scheven2}.

\section{Background and Definitions}
Let $f \in W^{2,2}(\Omega^m,N^n)$ with $\Omega \subset \Real^m$ and $N$ a smooth and compact manifold with $N \subset \Real^\ell$. Define the biharmonic energy functional \begin{equation}\label{bieq}
E(u):= \int_\Omega |\Delta f|^2 \, dx.
\end{equation}
The map $f$ is called \emph{weakly biharmonic} if
\begin{equation}\label{Ediff}
\frac \partial{\partial t}\bigg|_{t=0}E(f_t)=0
\end{equation}for all $X \in C^\infty_0(\Omega,\Real^\ell)$ where $f_t:= \pi_N(f+tX)$ and $\pi_N:U \to N$ is the nearest point retraction of a neighborhood $N \subset U$ onto $N$. The map $f$ is \emph{stationary biharmonic} if \eqref{Ediff} also holds for all $f_t(x):=f(x+t\xi(x))$ where $\xi \in C_0^\infty(\Omega,\Real^m)$. We define $f$ as \emph{minimizing biharmonic} if
\[
E(f) \leq E(v)
\]for all $v\in W^{2,2}(\Omega,N)$ such that $v=f$ on $\Omega \backslash K$ for some compact $K \subset \Omega$.

We will be interested in biharmonic maps with uniform energy bounds and for convenience define the following space.\begin{defn}For $\Lambda>0$ define the space $H^2_{\Lambda}(B_R(x),N)$ such that if $f:B_R(x)\subset \Real^m \to N^n$ and $f \in H^2_{\Lambda}(B_R(x),N)$ then
\begin{equation}\label{UEB1}
\int_{B_R(x)} |\nabla f|^2 + |\nabla^2 f|^2<\Lambda,
\end{equation}
\begin{equation}\label{UEB}
R^4\int_{B_R(x)} |\Delta f|^2dy + R^3\int_{\partial B_R(x)} \left(\partial_X|\nabla f|^2 + 4|\nabla f|^2- 4 \frac{|\partial_X f|^2}{|y-x|^2}\right) d\mathcal H^{m-1} \leq R^m\Lambda.
\end{equation}Here $\partial_X f = (y_i-x_i) \partial_i f$.
\end{defn}
Notice that we could replace the bound \eqref{UEB1} by an equivalent bound on the energy $E(f)$.

The proof of the Minkowski dimension estimate relies on a rigidity lemma, Lemma \ref{QRlemma}, which demonstrates that small changes in a monotonic quantity implies almost-homogeneity on a fixed scale.
Many authors have derived a monotonicity formula for stationary biharmonic maps \cite{Angelsberg,ChangWangYang,Scheven,Wang2}.

Let $f:B_R(0) \subset \Real^m \to N$ be a stationary biharmonic map with $f \in H^2_\Lambda(B_R(0),N)$. For any $x \in B_{R/4}(0)$ and a.e. $0<r<R/4$, the expression 
\begin{align}\label{ThetaDef}
\Theta_f(x,r):= &r^{4-m} \int_{B_r(x)} |\Delta f|^2dy \\&  \notag+ r^{3-m} \int_{\partial B_r(x)}\left( \partial_X|\nabla f|^2 + 4|\nabla f|^2 -4\frac{|\partial_X f|^2}{|y-x|^2} \right)d\mathcal H^{m-1}
\end{align} is well-defined and bounded above by $\Lambda$. 
Moreover, for all $x\in B_{R/4}(0)$ and a.e. $0<\rho<r<R/4$,
\begin{equation}\label{AnnularMonotone}
\Theta_f(x,r) - \Theta_f(x,\rho) = 4 \int_{B_r(x)\backslash B_\rho(x)} \left(\frac{|\nabla \partial_X f|^2}{|y-x|^{m-2}} + (m-2) \frac{|\partial_X f|^2}{|y-x|^{m}}\right) dy.
\end{equation}
The rigidity lemma demonstrates that when the quantity \eqref{AnnularMonotone} is sufficiently small on a fixed scale, a fixed dilation of the map $f$ is almost-homogeneous. We define higher homogeneity to record the largest linear subspace on which the map $f$ is constant.
\begin{defn}\label{homdef}
A measurable map $f:\Real^m \to N$ is \emph{$k$-homogeneous at $y$} with respect to the $k$-plane $V^k \subset \Real^m$ if
\begin{enumerate}
\item $f(y+ \lambda x) = f(y+x)$ for all $\lambda>0$ and $x \in \Real^m$,
\item $f(x+z) = f(x)$ for every $x \in \Real^m$ and $z \in V^k$.
\end{enumerate}If $y = 0$ we simply say that $f$ is \emph{$k$-homogeneous}.
\end{defn}\noindent Notice that a $0$-homogeneous map is radially constant and an $m$-homogeneous map is constant.

We will be interested in the homogeneity of quantitative blow-ups of the map $f$ at a fixed point. Presume that $f:B_4(0) \to N$. 
For $y \in B_2(0)\subset\Real^m$ and $0<r<2$, let $T_{y,r}f:B_{r^{-1}}(0) \to N$ such that
\[
T_{y,r}f(x) := f(r(x-y)).
\]A map $T_yf:\Real^m \to N$ is a \emph{tangent map} of $f$ at $y$ if there exists a sequence $r_i \to 0$ such that
\[
\fint_{B_1(0)}
|T_yf-T_{y,r_i}f|^2 \to 0.
\]Note that while a tangent map may not be unique, the monotonicity formula implies that tangent maps to stationary $f$ are always  $0$-homogeneous. 

Recall that $x \in \mathcal S^k(f)$ if and only if no tangent map $T_xf$ at $x$ is $(k+1)$-homogeneous.
As we are interested in quantitative results, we will not consider the homogeneity of tangent maps of $f$ but instead will quantify almost homogeneity of almost tangent maps of $f$. 
\begin{defn}
A measurable map $f:B_{2r}(x) \subset \Real^m \to N$ is \emph{$(\varepsilon,r,k)$-homogeneous at $x$} if there exists a $k$ homogeneous map $h:\Real^m \to N$ such that
\[
\fint_{B_1(0)}
|T_{x,r}f-h|^2 \leq \varepsilon.
\]
\end{defn}If $h$ is $k$-homogeneous with respect to $V^k \subset \Real^m$ we denote
\begin{equation}\label{fVplane}
V^k_{f,x} := B_r(x) \cap (V^k+x).
\end{equation}Notice that if $f$ is $(0,r,k)$-homogeneous at $x$ then $f|_{V^k_{f,x}}$ is a.e. constant.
\begin{defn}\label{qssdef}Let $f :B_4(0)\to N$ be a measurable map. 
For each $\eta>0$ and $0<r<1$ the \emph{$k^{th}$ quantitative singular stratum $\mathcal S^k_{\eta,r}(f) \subset \Real^m$} is the set
\begin{align}
\mathcal S^k_{\eta,r}(f):= \bigg\{ x \in B_2(0): &\fint_{B_1(0)}
|T_{x,s}f-h|^2>\eta \text{ for all } r \leq s \leq 1 \\&\notag \text{ and } (k+1)-\text{homogeneous maps } h \bigg\}.
\end{align}
\end{defn}By definition $x \in S^k_{\eta,r}(f)$ if and only if $f$ is not $(\eta,s,k+1)$-homogeneous for all $r \leq s \leq 1$. Further, we have the following relations:
\begin{equation*}
\mathcal S^{k}_{\eta,r}(f) \subset \mathcal S^{k'}_{\eta',r'}(f) \,\,\text{ if } k\leq k', \,\eta' \leq \eta,\, r \leq r',
\end{equation*}
\begin{equation}\label{SingInt}
\mathcal S^k(f) = \bigcup_\eta \bigcap_r \mathcal S^k_{\eta,r}(f).
\end{equation}

\section{Stratification Estimates}
The method of quantitative stratification has been applied in a number of settings. In addition to the previously mentioned results, quantitative stratification has been used to improve regularity results for parabolic problems \cite{CHNMCF,CHNHF}, for a more general class of elliptic problems \cite{CNV}, and recently for $p$-harmonic maps \cite{NVV}. In a recent preprint \cite{FMS}, Focardi, Marchese, and Spadaro outline an abstract approach to quantitative stratification that can be applied under the presumption of a structural hypothesis. The hypotheses are satisfied by the existence of a monotonicity formula and a cone-splitting principle. We present the quantitiative stratification argument for biharmonic maps in its entirety for clarity of exposition and to demonstrate that the argument is easily modified if the monotonicity formula only holds almost everywhere. 

Our first result gives a volume bound on the tubular neighborhood of the quantitative singular strata. 

\begin{thm}\label{minkowskibound}
Let $f:B_4(x)\subset \Real^m \to N^n$ denote a stationary biharmonic map with $f\in H^2_\Lambda(B_4(x),N)$. Then for all $\eta >0$ there exists $C=C(m, N, \Lambda, \eta)$ such that for any $0<r<1$,
\begin{equation}\label{TrSS}
\mathrm{Vol}(T_r(\mathcal S^k_{\eta,r}(f))\cap B_1(x)) \leq Cr^{m-k-\eta}
\end{equation}where $T_r(A)$ denotes the $r$ tubular neighborhood of a set $A$. 
\end{thm}

\begin{rem}
Since all of our results are local, the arguments of this paper extend to maps from a smooth, closed manifold with uniform bounds on curvature and injectivity radius. The arguments then follow the same outline, modulo lower order terms that arise from the geometry of the domain.
\end{rem}

Recall that the Minkowski dimension of a set $A$ is determined by the \emph{$r$-content} of the set. For all $r \leq 1$ and all $\eta >0$, if the number of closed metric balls of radius $r$ in a minimal covering of $A$ is bounded by $C_\eta r^{-(d+\eta)}$, then $\dim_{\mathrm{Min}}A \leq d$. (Throughout the paper we let $\dim_{\mathrm{Min}}A$ denote the Minkowski dimension of the set $A$.) Equivalently, we can define
\[
\dim_{\mathrm{Min}}A := \inf\left\{s: \limsup_{\delta \to 0}\frac{\mathrm{Vol}(T_\delta(A))}{\delta^{m-s}}<\infty\right\}.
\]
As an immediate consequence of \eqref{SingInt} and \eqref{TrSS}, $\dim_{\mathrm{Min}}(\mathcal S^k) \leq k$. Moreover, since $\dim A \leq \dim_{\mathrm{Min}}A$, we recover the Hausdorff bound on the singular strata \eqref{SingDim}.

The proof  of Theorem \ref{minkowskibound} follows the same general strategy of previous quantitative stratification papers \cite{CNCPAM,CNInvent,CHNMCF,CHNHF,CNV}. We begin by proving a quantitative version of the fact that tangent cones to stationary biharmonic maps are $0$-homogeneous. We then use this lemma to decompose $B_1(x)$ into a union of sets where the sets encode the behavior of points on various scales. The uniform energy bound for $f$ implies that the number of sets grows at most polynomially in the number of scales. Using a quantitative cone-splitting lemma, we demonstrate that points in $\mathcal S^j_{\eta,r}(f)$ with ``good'' behavior on the same scale roughly line up in a neighborhood of some subspace $\Real^j \subset \Real^m$. We then find a suitable covering on which we compute the volume of $\mathcal S^j_{\eta,r}(f) \cap B_1(0)$.

We first recall the cone-splitting principle for measurable maps: If $h:\Real^m \to N^n$ is $k$-homogeneous at $y$ with respect to the $k$-plane $V^k$, $h$ is $0$-homogeneous at $z \in \Real^m$, and $z \notin y+V$, then $h$ is $(k+1)$-homogeneous at $y$ with respect to the $(k+1)$-plane $\mathrm{span}\{z-y, V^k\}$.
We now state a quantitative version of cone splitting that we will use. 
 \begin{lem}[Cone-Splitting Lemma] \label{conesplitting}Given $\eta, \tau>0$ there exists $\varepsilon=\varepsilon(m,N,\Lambda, \eta,\tau)>0$ such that the following holds: Let $f:B_{4}(0) \to N$ with $\|f\|_{W^{2,2}(B_{4}(0))}\leq \Lambda$. If, in addition, for $x \in B_1(0)$ and $0<r<1$,
\begin{enumerate}
\item $x \in \mathcal S^k_{\eta,r}(f)$ and
\item $f$ is $(\varepsilon,2r,0)$-homogeneous at $x$
\end{enumerate}then there exists some $k$-plane $V$ such that
\[
\{y \in B_r(x) : f \text{ is }(\varepsilon,2r,0)-\text{homogeneous at }y\} \subset T_{\tau r}(V).
\]
\end{lem}Recall that $x \in \mathcal S^k_{\eta,r}(f)$ iff $f$ is not $(\eta,s, k+1)$-homogeneous at $x$ for all $r \leq s \leq 1$.
\begin{proof}
Arguing by contradiction we assume without loss of generality that $x=0$ and $r=1$. Thus, given $\varepsilon, \tau >0$ there exists a sequence $\{f_i\}$ with $\|f_i\|_{W^{2,2}(B_{4}(0))}\leq \Lambda$ such that $0 \in \mathcal S^k_{\eta,1}(f_i)$ for all $i$, $f_i$ is $(1/i,2,0)$-homogeneous at $0$ and there exist points $\{x_1^i, x_2^i,\dots, x_{k+1}^i\}\subset B_1(0)$ satisfying the following two conditions:
\begin{itemize}
\item $f_i$ is $(1/i,2,0)$-homogeneous at each $x_j^i$ for $j=1, \dots, k+1$,
\item for all $j=1, \dots, k+1$, $\mathrm{dist}(x_j^i,\mathrm{span}\{0,x_1^i, \dots, x_{j-1}^i\}) \geq \tau$.
\end{itemize} After passing to a subsequence, there exists a map $f$ such that $f_i \to f$ strongly in $W^{1,2}(B_4(0))$ and weakly in $W^{2,2}$. After passing to a further subsequence, there exist points $\{x_1, \dots, x_{k+1}\} \subset \overline {B_1(0)}$ such that $f$ is $(0,2,0)$-homogeneous at $0$ and at each $x_j$, $j=1, \dots, k+1$. Moreover the distance relations are preserved so for all $j=0, \dots, k+1$, and $x_0=0$, $\mathrm{dist}(x_j,\mathrm{span}\{x_0,x_1, \dots, x_{j-1}\}) \geq \tau$.

Notice that $f=\phi_j$ a.e. $B_2(x_j)$ for some $0$-homogeneous $\phi_j$. Moreover $\phi_j=\phi_m$ on $B_2(x_j)\cap B_2(x_m)$. Taken together, these statements imply that for every $j=1,\ldots,k+1$ 
$$\phi_0(0)= \phi_j(0)= \phi_j(x_j/2)=\phi_0(x_j/2).$$
Thus $\phi_0$ is constant in $\mathrm{span}\{0, x_1, \dots, x_{k+1}\}$, a $k+1$-dimensional subspace of $\Real^m$. Since $f|_{B_1(0)}=\phi_0$ a.e., it follows that $f$ is $(0,1,k+1)$-homogeneous at $0$. The strong convergence of $f_i \to f$ in $W^{1,2}$ is enough to give a contradiction.

\end{proof}

Next we demonstrate that if a stationary $f$ has sufficiently small monotonic difference on the boundary of an annulus, then $f$ is almost radially constant. 
\begin{lem}[Quantitative Rigidity]\label{QRlemma}
Let $f:B_4(0) \subset \Real^m\to N^n$ be a stationary biharmonic map with $f \in H^2_\Lambda(B_4,N)$. Then for every $\varepsilon >0$ and $0<\gamma<1/2$ there exist $\delta = \delta(\gamma,\varepsilon, m, N, \Lambda)>0$ and $q=q(\gamma,\varepsilon,m,N,\Lambda)\in \mathbb N$ such that for $r \in (0,1/2)$, if $\Theta_f(0,2r) - \Theta_f(0,\gamma^q r)$ is well-defined and
\begin{equation}
\Theta_f(0,2r) - \Theta_f(0,\gamma^q r) \leq \delta,
\end{equation}then $f$ is $(\varepsilon,2r,0)$-homogeneous at $0$.
\end{lem}

\begin{proof}
Assume there exist $\varepsilon>0$ and $0<\gamma<1/2$ for which the statement is false. Again we assume that $r=1$. Then there exists a sequence of stationary biharmonic maps $f_i:B_4(0) \to N^n$ with $f_i \in H^2_\Lambda(B_4,N)$ and $\Theta_{f_i}(0,2) - \Theta_{f_i}( 0,\gamma^i ) \leq i^{-1}$ but each $f_i$ is not $(\varepsilon, 2,0)$-homogeneous. As $f_i \to f$ strongly in $W^{1,2}(B_3)$ and $|\partial_Xf| = |s \partial_s f|$ for $s(x): = |x|$, \eqref{AnnularMonotone} implies that
\[
\int_{B_1(0)}  \frac{|\partial_X f|^2}{|y|^2} \,dy \leq \liminf \int_{B_1(0)} \frac{|\partial_X f_i|^2}{|y|^2} \, dy \leq \liminf \int_{B_1(0)}  \frac{|\partial_X f_i|^2}{|y|^{m}}\, dy\leq 0.
\]Thus, $|\partial_s f| =0$ a.e. and we conclude that  $f$ is $(0,2,0)$-homogeneous. By the strong convergence of a subsequence, $f_i$ to $f$ in $L^2$, there exists $i$ sufficiently large such that
\[
\fint_{B_{2}(0)}|f_i-f|^2 < \varepsilon.
\]Thus, $f_i$ is $(\varepsilon,2,0)$-homogeneous which gives a contradiction.
\end{proof}

\begin{defn}
For a stationary biharmonic map $f\in H^2_\Lambda(B_4(0),N)$, $x \in B_1(0)$ and $0\leq s <t<1$, define
\[
\mathcal W_{s,t}(x,f):=\Theta_f(x,t) - \Theta_f(x,s)\geq 0.
\]This quantity is well-defined for all $x \in B_1(0)$ and a.e. $0 \leq s<t<1$.
\end{defn}Notice that, by \eqref{AnnularMonotone}, for $(s_1, t_1), (s_2, t_2)$ with $t_1 \leq s_2$, 
\[
\mathcal W_{s_1,t_1}(x,f)+ \mathcal W_{s_2,t_2}(x,f) \leq \mathcal W_{s_1,t_2}(x,f)
\] with equality iff $t_1=s_2$.

For all $x \in B_1(0)$ and $0<\gamma<1/2$, fix $q \in \mathbb Z^+$. For all $j \in \mathbb Z^+$, the quantity $\mathcal W_{s,t}(x,f)$ is well-defined for at least one pair $(s_j,t_j) \in [\gamma^{j+q+1/2},\gamma^{j+q}] \times [\gamma^{j},\gamma^{j-1/2}]$. We let $Q$ be the number of positive integers $j$ such that
\[
\mathcal W_{s_j,t_j}(x,f)>\delta.
\]We note that
\begin{equation}\label{Nbound}
Q \leq (q+3)\Lambda \delta^{-1}.
\end{equation}If not, there would exist $\Lambda \delta^{-1}$ disjoint annuli defined by $(s_i, t_i)$ with $\mathcal W_{s_i,t_i}(x,f)>\delta$. But then
\[
\Lambda < \sum_{i=1}^Q\mathcal W_{s_i,t_i}(x,f) \leq \mathcal W_{0,1}(x,f)\leq \Theta_f(0,4)
\]which contradicts \eqref{UEB}.

\begin{defn}
For a fixed $0<\gamma<1/2$ and $q \in \mathbb Z^+$, for $j\in \mathbb Z^+$, let $A_j$ denote the region in $\Real^2$ such that
\[
A_j:=\{(s,t)|\, \gamma^{j+q+1/2}\leq s \leq\gamma^{j+q}, \gamma^{j}\leq t \leq \gamma^{j-1/2}\}.
\]Notice that for $(s,t) \in A_j$, $s <t$.
\end{defn}

Following \cite{CNCPAM,CNInvent,CHNMCF,CHNHF}, for each $x \in B_1(0)$, we define a sequence $\{T_j(x)\}_{j \geq 1}$ with values in $\{0,1\}$ in the following manner. For a fixed $0<\gamma<1/2$ and $q\in \mathbb Z^+$, for each $j \in \mathbb Z^+$ define
\[
T_j(x) = \left\{ \begin{array}{ll}1,& \text{if } \mathcal W_{s,t}(x,f) > \delta \text{ for all }(s,t) \in A_j \text{ for which } \mathcal W_{s,t}(x,f) \text{ is defined}, \\
0, & \text{if } \mathcal W_{s,t}(x,f) \leq \delta \text{ for at least one pair }(s,t) \in A_j .\end{array} \right.
\]

Thus, \eqref{Nbound} implies that there exist at most $Q$ nonzero entries in the sequence.
For each $\beta$-tuple $\{T_j^\beta\}_{1 \leq j \leq \beta}$, we define the set
\[
E_{T^\beta}(f)=\{ x \in B_1(0)|T_j(x) =T_j^\beta \text{ for } 1 \leq j \leq \beta\}.
\]While a priori there are $2^\beta$ possible sets $E_{T^\beta}(f)$, the bound on the number of non-zero entries and the fact that $\binom{\beta}{Q}  \leq \beta^{Q}$ implies that the number of nonempty set $E_{T^\beta}(f)$ is at most $\beta^{Q}$. Thus, for each $\beta$, we have decomposed $B_1(0)$ into at most $\beta^{Q}$ non-empty sets $E_{T^\beta}(f)$.

\begin{lem}[Covering Lemma]
There exists $c_0(m)<\infty$ such that each set $\mathcal S^j_{\eta,\gamma^\beta}(f) \cap E_{T^\beta}(f)$ can be covered by at most $c_0(c_0\gamma^{-m})^{Q}(c_0 \gamma^{-j})^{\beta-Q}$ balls of radius $\gamma^\beta$.
\end{lem}
\begin{proof}For fixed $\eta, \gamma$, let $\tau=\gamma$ and choose $\varepsilon$ as in Lemma \ref{conesplitting}. For this $\varepsilon, \gamma$, by Lemma \ref{QRlemma} there exist $\delta>0$ and $q \in \mathbb Z^+$ such that if $\Theta(x,2\gamma^{j})-\Theta(x,\gamma^{j+q} ) \leq \delta$ then $f$ is $(\varepsilon,2\gamma^{j},0)$-homogenous at $x$. Fix this $\delta,q$ throughout the proof and define $A_j$, $T_j(x)$ accordingly.

We now determine the covering. We begin by choosing, for $\beta=0$, a minimal covering of $\mathcal S^j_{\eta,1}(f)\cap B_1(0)$ by balls of radius $1$ with centers in $\mathcal S^j_{\eta,1}(f)\cap B_1(0)$. Next, note that $\mathcal S^j_{\eta, \gamma^{\beta + 1}}(f) \subset \mathcal S^j_{\eta, \gamma^\beta}(f)$. Moreover, if we let $T^\beta$ denote the $\beta$-tuple obtained by dropping the last entry from $T^{\beta+1}$ then we immediately have $E_{T^{\beta+1}}(f)\subset E_{T^\beta}(f)$. 

We determine the covering recursively. For each ball $B_{\gamma^\beta}(x)$ in the covering of $\mathcal S^j_{\eta,\gamma^\beta} (f)\cap E_{T^\beta}(f)$, take a minimal covering of $B_{\gamma^\beta}(x)\cap \mathcal S^j_{\eta,\gamma^\beta}(f) \cap E_{T^\beta}(f)$ by balls of radius $\gamma^{\beta+1}$. There are now two possibilities. If $T^{\beta}_{\beta}=1$, then every $x \in E_{T^\beta}(f)$ has $T_\beta(x)=1$ and every $y \in B_{\gamma^\beta}(x)\cap E_{T^\beta}(f)$ satisfies the property $\mathcal W_{s,t}(y,f) > \delta$ { for all }$(s,t) \in A_{\beta}$ for which the quantity is defined. Thus we can only bound the number of balls geometrically to get an upper bound on the covering by
\[
c(m) \gamma^{-m}.
\]On the other hand, if $T_{\beta}^\beta=0$, we can do better. Then every $x \in E_{T^\beta}(f)$ has $T_\beta(x)=0$ and $\mathcal W_{s,t}(x,f) \leq \delta${ for at least one pair }$(s,t) \in A_{\beta}$. By choice of $\delta, q$, this implies that $f$ is $(\varepsilon, 2\gamma^{\beta},0)$ homogeneous at $x$. Since $x \in \mathcal S^j_{\eta,\gamma^\beta}(f)$, we can apply Lemma \ref{conesplitting} to conclude that every $y \in E_{T^\beta}(f)\cap B_{\gamma^{\beta}}(x)$ is contained in a $\gamma^{\beta+1}$ tubular neighborhood of some $j$ dimensional plane $V$. Therefore, in this case we can cover the intersection with the smaller number of balls
\[
c(m)\gamma^{-j}.
\]Given any $\beta>0$ and $E_{T^\beta}(f)$, the number of times we need to apply the weaker estimate is bounded above by $Q$. Thus, the proof is complete.
\end{proof}

We are now ready to prove Theorem \ref{minkowskibound}.

\begin{proof}[Proof of Theorem \ref{minkowskibound}]
Choose $\gamma<1/2$ such that $\gamma\leq c_0^{-2/\eta}$, where $c_0$ is as in Lemma $3.7$. Then $c_0^\beta \leq (\gamma^\beta)^{-\eta/2}$ and since exponentials grow faster than polynomials,
\[
\beta^Q\leq c(Q)c_0^\beta \leq c(\eta,m,Q) (\gamma^\beta)^{-\eta/2}.
\]Since
\[
\mathrm{Vol}(B_{\gamma^\beta}(x))= \omega_m\gamma^{\beta m},
\]and $B_1(0)$ can be decomposed into $\beta^Q$ sets $E_{T^\beta}(f)$ for any $\beta$,
\begin{align*}
\mathrm{Vol}(\mathcal S^j_{\eta,\gamma^\beta}\cap B_1(0)) &\leq \beta^Q c_0\left[(c_0\gamma^{-m})^Q(c_0\gamma^{-j})^{\beta-Q}\right]\omega_m \gamma^{\beta m}\\
&\leq c(m,Q,\eta) \beta^Q c_0^\beta(\gamma^\beta)^{m-j}\\
&\leq c(m,Q,\eta)(\gamma^\beta)^{m-j-\eta}.
\end{align*}Thus, for any $0<r<1$, there exists $\beta>0$ such that $\gamma^{\beta+1} \leq r < \gamma^\beta$. Then
\begin{align*}
\mathrm{Vol}(\mathcal S^j_{\eta,r} \cap B_1(0)) &\leq \mathrm{Vol}(\mathcal S^j_{\eta,\gamma^\beta} \cap B_1(0))\\
&\leq c(m,Q,\eta)(\gamma^\beta)^{m-j-\eta}\\
& \leq c(m,Q,\eta)(\gamma^{-1}r)^{m-j-\eta}\\
&\leq c(m,\eta,N,\Lambda) r^{m-j-\eta}.
\end{align*}
\end{proof}

\section{Improved Regularity Results}

In this section we determine improved regularity results for minimizing biharmonic maps and stationary biharmonic maps into special targets. The regularity results do not require a small energy hypothesis and the $L^p$ estimates are both dimension independent and sharp. 
The $L^p$ bounds on $\nabla^k f$, $k=1,\dots, 4$, are a corollary of a much stronger estimate on the $L^p$ bounds on the reciprocal of the regularity scale function. Finally, we demonstrate via a covering argument that when the singular set is isolated, there exists a uniform bound on the number of singular points in a compact subset. 

For certain targets, we demonstrate higher regularity. The targets of interest have one of the following two forms.
\begin{defn}
We say that $N$ satisfies:
\begin{itemize}
\item {\bf Condition M} if there exists some $k_0 \geq 4$ such that for all $4 \leq k \leq k_0$, there are no $0$-homogeneous, non-constant minimizing biharmonic maps $v \in C^\infty(\Real^{k+1} \backslash \{0\},N)$.\vskip .1in
\item {\bf Condition S} if there exists $k_0\geq 4$ such that for all $4 \leq k \leq k_0$, there are no $0$-homogeneous, non-constant biharmonic maps $v \in C^{\infty}(\Real^{k+1}\backslash\{0\},N)$.
\end{itemize}
\end{defn}
The $L^p$ estimates on derivatives of $f$ are determined by first demonstrating that the region on which $f$ does not possess scale invariant $C^4$ bounds is small. 
We define a function that captures the largest radius about a point on which the function does possess scale invariant $C^4$ bounds. 
\begin{defn}\label{regscale}
Let $r_{0,f}(x)$ be the largest $r$ such that $f\in C^4(B_r(x))$. We denote the \emph{regularity scale of $f$ at $x$} as
\[
r_f(x):=\max\{0\leq r\leq r_{0,f}(x): \sup_{B_r(x)} r|\nabla f| + r^2 |\nabla^2 f|+r^3|\nabla^3 f| + r^4|\nabla^4f| \leq 1\}.
\]
\end{defn}
Notice that if $r_f(x)=r>0$, then $f \in C^\infty(B_{\frac r2}(x))$ (see \cite{ChangWangYang,Wang,Wang2}).

Given a measurable map $f:\Omega \subset\Real^m \to N^n$ and any $r>0$ define the set
\[
\mathcal B_r(f):= \{x \in \Omega: r_f(x) \leq r\}.
\]
The volume of the set of points in $\mathcal B_r(f)$, for any fixed $r>0$, will be controlled using the bound we determined in Theorem \ref{minkowskibound}. 
\begin{thm}\label{MinMinkowski}
Let $f:B_4(x)\subset \Real^m \to N^n$ be a biharmonic map with $f \in H^2_\Lambda(B_4(x),N)$. 
\begin{itemize}
\item If $f$ is minimizing then for all $\eta>0$, there exists $C=C(m,N,\Lambda,\eta)$ such that for any $0<r<1$
\[
\mathrm{Vol}(T_r(\mathcal B_r(f))\cap B_1(x)) \leq Cr^{5-\eta}.
\]For minimizing biharmonic maps, this bound implies that
\[
\dim_{\mathrm{Min}}\mathcal S(f) \leq m-5.
\]
\item
If $f$ is minimizing and $N$ satisfies Condition M or if $f$ is stationary and N satisfies Condition S, then for all $\eta>0$, there exists $C=C(m,N,\Lambda,\eta)$ such that for any $0<r<1$
\[
\mathrm{Vol}(T_r(\mathcal B_r(f))\cap B_1(x)) \leq Cr^{k_0+2-\eta}.
\]
In particular, with the additional hypothesis on the target manifold,
\[
\dim_{\mathrm{Min}}\mathcal S(f) \leq m-(k_0+2).
\]
\end{itemize}
\end{thm}

We bound the $L^p$ norm for the reciprocal of $r_f$ by considering the $L^p$ norm over sets of points with regularity scale between $r_i:=2^{-i}$ and $r_{i+1}$. The volume estimate above allows us to conclude that the upper bounds on these integrals are summable in $i$. 
\begin{cor}\label{MinCor}
Let $f:B_4(x)\subset \Real^m \to N^n$ be a biharmonic map with $f \in H^2_\Lambda(B_4(x),N)$. 
\begin{itemize} 
\item If $f$ is minimizing then for all $1\leq p<5$, there exists $C=C(m,N,\Lambda,p)$ such that
\[
\int_{B_1(x)} \sum_{\ell=1}^4|\nabla^\ell f|^{p/\ell}  \leq 4\int_{B_1(x)} r_f^{-p} <C.
\]
\item If $f$ is minimizing and $N$ satisfies Condition M or $f$ is stationary and $N$ satisfies Condition S, then for all $1\leq p<k_0+2$, there exists $C=C(m,N,\Lambda,p)$ such that
\[
\int_{B_1(x)} \sum_{\ell=1}^4 |\nabla^\ell f|^{p/\ell}\leq 4\int_{B_1(x)} r_f^{-p} <C.
\]
\end{itemize}
\end{cor}

\begin{rem}In the harmonic setting, the sharpness of the result was shown by considering the minimizing, harmonic projection map $f(x) = \frac x{|x|}$. We use the same map to verify the optimality of the first result here for $m=5$. Indeed, the map $f:B_1^5 \to \mathbb S^{4}$ with $f(x) = \frac x{|x|}$ is also minimizing biharmonic. Since such a map is harmonic, it is clearly an \emph{intrinsic} biharmonic map (i.e., a critical point for the energy $\int_\Omega |(\Delta f)^T|^2$). For our purposes, it is important to recognize that the map is also a minimizer for the biharmonic energy defined in \eqref{bieq}. (See the appendix of \cite{HongWang} for a more thorough treatment of this fact.) As $|\nabla \frac {x}{|x|}|\sim 1/r$, $|\nabla f|$ is clearly not in $L^5$.
\end{rem}
Finally, following the work of \cite{NVV}, we show that in settings where $\dim \mathcal S(f) =0$, there exists a uniform upper bound on the number of singular points on a compact set.
\begin{thm}\label{counting}
Let $f:B_4(x)\subset \Real^m \to N^n$ denote a biharmonic map with $f\in H^2_\Lambda(B_{4}(x),N)$.
\begin{itemize}
\item If $f$ is minimizing and $m=5$ then
\[
|\mathcal S(f) \cap B_1(x)| \leq C(N, \Lambda).
\]
\item If $f$ is minimizing and $N$ satisfies Condition $M$ or if $f$ is stationary and $N$ satisfies Condition S, then for $m=k_0+2$
\[
|\mathcal S(f) \cap B_1(x)| \leq C(N, \Lambda).
\]

\end{itemize}
\end{thm}

\subsection{Compactness Results}All of the proofs are based on compactness results, some of which are only implicit in the literature. Since some results are modifications of their original form and others are not stated, at least to our knowledge, we state the compactness lemmas in the forms we will need and provide some indication of the proofs. We also note that we make frequent use of Theorems 1.5 and 1.6 of \cite{Scheven} which provide compactness theorems for sequences of minimizers and stationary maps. 

While Theorem 1.5 of \cite{Scheven} does not imply that the space of minimizing biharmonic maps is compact, results of that paper do imply the following compactness result for minimizers.
\begin{lem}\label{homogeneouslimits}
Let $f_i \to f$ in $W^{2,2}(B_2^m,N)$ such that each $f_i$ is minimizing. Then
\begin{itemize}
\item if $m=5$ and $f$ is $0$-homogeneous then $f$ is minimizing in $B_2$.
\item if $N$ satisfies Condition M and $f$ is $(m-k)$-homogeneous for any $k \leq k_0+2$, then $f$ is minimizing in $B_2$.
\end{itemize}
\end{lem}
\begin{proof}We only point out here that the proof of Lemma 4.2 in \cite{Scheven} is sufficient to prove that $f$ is minimizing. The lemma in \cite{Scheven} is concerned with tangent maps to minimizers, but the proof relies only on the homogeneity of tangent maps and the fact that the singular set of a tangent map is a linear subspace of $\Real^m$. 

For $m=5$, any $0$-homogeneous map must be smooth away from the origin and thus Scheven's comparison construction can be directly extended to prove that $f$ is minimizing.

When $f$ is $(m-k)$-homogeneous, the restriction $\hat f:B_2^{k} \times \{0\}$ is $0$-homogeneous. Moreover, Theorem 1.1 of \cite{Scheven} implies that as $N$ satisfies Condition $M$, $\hat f$ is smooth away from the origin. Thus, the same comparison construction proves that $f, \hat f$ are minimizers in spaces $B_2^m, B_2^{k}$ respectively. 
\end{proof}

The following lemma provides an $\varepsilon$-regularity result.
\begin{lem}\label{lemmatechnical}
Let $f:B_4(0)\subset \Real^m \to N^n$ denote a biharmonic map with $\|f\|_{W^{2,2}}\leq \Lambda$. If $f$ is minimizing or $f$ is stationary and $N$ satisfies Condition $S$ then there exists $\varepsilon_0=\varepsilon_0(m,N,\Lambda)>0$ such that if there exists $c \in N$ such that
\[
\fint_{B_2(0)}|f-c|^2 <\varepsilon_0
\]then $f \in C^\infty(B_1(0))$. 
\end{lem}
\begin{proof}
We proceed by contradiction. Suppose no such $\varepsilon_0$ exists. Then there exist $\|f_i\|_{W^{2,2}}\leq \Lambda$ and $c_i \in N$ such that
\[
\fint_{B_2(0)}|f_i-c_i|^2 <1/i
\]but each $f_i \notin C^\infty(B_1(0))$. By Theorems 1.5, 1.6 in \cite{Scheven}, a subsequence $f_i \to f$ strongly in $W^{2,2}(B_2(0))$. Extracting a further subsequence, there exists $c \in N$ such that 
\[
\|f_i - c\|_{W^{2,2}(B_2(0))} \leq \|f_i-f\|_{W^{2,2}(B_2(0))} + \|f-c\|_{W^{2,2}(B_2(0))} \to 0.
\]
Results of \cite{Wang,Struwe} imply that there exists $\varepsilon_1>0$ such that if $u$ is stationary biharmonic with
\[
2^{4-m} \int_{B_2} |\nabla^2 u|^2+ |\nabla u|^4 \,dx < \varepsilon_1 
\]then 
\[
\|u \|_{C^\infty(B_2)} \leq C(N).
\]Since $f_i \to c$ in $W^{2,2}(B_2,N)$, given $\varepsilon>0$, for any fixed $C>0$ there exists $I>0$ such that for all $i \geq I$,
\[
C\|f_i\|_{L^\infty(B_2)}^2\left(2^{4-m}\int_{B_2}|\nabla^2 f_i|^2\, dx+ 2^{2-m}\int_{B_2}|\nabla f_i|^2\,dx\right) < \varepsilon/ 2.
\]Choose $0<\varepsilon <\varepsilon_1$ above. Then by a Nirenberg interpolation inequality, for all $i \geq I$,
\[
2^{4-m}\int_{B_2}|\nabla f_i|^4 \leq C\|f_i\|_{L^\infty(B_2)}^2 \left( 2^{4-m}\int_{B_2}|\nabla^2 f_i|^2\, dx+ 2^{2-m}\int_{B_2}|\nabla f_i|^2\,dx \right)<\varepsilon_1/2.
\]Increasing $I$, if necessary, we conclude that for all $i \geq I$,
\[
2^{4-m} \int_{B_2} |\nabla^2 f_i|^2+ |\nabla f_i|^4 \,dx < \varepsilon_1 
\]and thus by \cite{Struwe,Wang} and Arz\'ela-Ascoli, $f_i$ subconverges to $c$ in $C^\infty(B_2,N)$. This implies the necessary contradiction.
\end{proof}

\subsection{Proof of regularity scale and $L^p$ estimates}

The main tool in the proof of the Minkowski bound on $\mathcal B_r(f)$ is a quantitative $\varepsilon$-regularity theorem. The classical assumption of small energy is replaced here by a homogeneity hypothesis. The theorem establishes that minimizing biharmonic maps that are almost homogeneous in enough directions must have a large regularity scale. For harmonic maps, minimizers with $(m-2)$ directions of almost homogeneity are sufficiently regular \cite{CNCPAM}. For biharmonic minimizers $(m-4)$ directions of almost homogeneity suffice.
\begin{thm}[$\varepsilon$-Regularity]\label{epsreg}
Let $f:B_4(0) \to N$ be a biharmonic map with $f \in H^2_\Lambda(B_4,N)$. 
\begin{itemize}
\item If $f$ is minimizing then there exists $\varepsilon=\varepsilon(m, N, \Lambda)>0$ such that if $f$ is $(\varepsilon, 2, m-4)$-homogeneous, then
\[
r_f(0) \geq 1.
\]
\item If $f$ is minimizing and $N$ satisfies Condition M or $f$ is stationary and $N$ satisfies Condition $S$, then there exists $\varepsilon=\varepsilon(m, N, \Lambda)>0$ such that if $f$ is $(\varepsilon, 2, m-k_0-1)$-homogeneous, then
\[
r_f(0) \geq 1.
\]
\end{itemize}
\end{thm}

The proof relies on the following lemma which is true for any $W^{2,2}$ map. 

\begin{lem}\label{epsreglemma}
For all $\varepsilon>0$ there exists $\delta=\delta(m,N,\Lambda, \varepsilon)>0$ such that if $f:B_4(0) \to N$ is a $(\delta,2,m-4)$-homogeneous map with $\|f\|_{W^{2,2}(B_4(0))}\leq \Lambda$, then $f$ is $(\varepsilon,2,m)$-homogeneous.
\end{lem}The proof is similar to the one found in \cite{CNCPAM}, though here we use the fact that $f \in W^{2,2}$ to allow for the weaker hypothesis of $(m-4)$-homogeneity.
\begin{proof}
We proceed by contradiction. Presume there exists an $\varepsilon>0$ for which the statement does not hold. Then there exists a sequence of $f_i:B_4(0) \to N$ with $\|f_i\|_{W^{2,2}} \leq \Lambda$ and each $f_i$ is $(i^{-1},2,m-4)$ homogeneous but not $(\varepsilon,2,m)$-homogeneous. There exists an $f_\infty\in W^{2,2}(B_4(0))$ such that a subsequence $f_{i} \to f_\infty$ strongly in $W^{1,2}(B_4)$. The convergence implies that $f_\infty$ is $(0,2,m-4)$-homogenous but not $(\varepsilon,2,m)$-homogeneous. We first establish that
\[
\int_{B_2(0)} |\nabla^2 f_\infty|^2 =0.
\]Suppose not. Then the a.e. radial invariance of $f_\infty$ implies that there exists some $c>0$ and $x \in \Real^m \cap \partial B_1(0)$ such that
\[
\int_{B_{\frac 12}(x)}|\nabla^2 f_\infty|^2 =c.
\]Since $f_\infty$ is $0$-homogeneous a.e., $\nabla^2f_\infty(x) = \lambda^{-2}\nabla^2 f_\infty(\lambda x)$ a.e. and thus for all $j \in \mathbb N$,
\begin{equation}\label{homenergy}
\int_{B_{4^{-j}/2}(x/4^j)}|\nabla^2 f_\infty|^2 = c(4^{-j})^{m-4}.
\end{equation}Notice that the balls $B_{ 4^{-j}/2}(x/4^j)$ are mutually disjoint. 

Since $f_\infty$ is $(0,2,m-4)$-homogeneous, there exists $c_1>0$ such that for each $j \in \mathbb N$ there exists a collection of at least $c_1(4^j)^{-m+4}$ mutually disjoint balls in $B_2(0)$ of radius $4^{-j}/2$ such that \eqref{homenergy} holds on each ball. Moreover, when $j_1 \neq j_2$, the collections are obviously disjoint. Therefore,
\[
\int_{B_2(0)} |\nabla^2 f_\infty|^2\geq \sum_{j=1}^\infty c(4^{-j})^{m-4}c_1(4^j)^{-m+4}=\infty.
\]But this contradicts the uniform $W^{2,2}$ bound on the $f_i$ and the lower semi-continuity of the energy. Therefore, $\int_{B_2(0)} |\nabla^2 f_\infty|^2=0$.

Now, by the Poincar\'e inequality,
\[
\int_{B_2(0)}|\nabla f_\infty - A|^2 \leq \int_{B_2(0)}|\nabla^2f_\infty|^2 =0
\]and thus $\nabla f_\infty$ is $(0,2,m)$-homogeneous and $f_\infty(x) = Ax+b$ a.e. Since $f_\infty$ is $(0,2,0)$-homogeneous at $0$, $Ax = \frac 12 Ax$ for a.e. $x \in B_2$. That is, $A x= 0$ a.e. and thus $f_\infty$ is $(0,2,m)$-homogeneous at $0$. The strong convergence of $f_i$ to $f_\infty$ in $W^{1,2}(B_2)$ implies the necessary contradiction.
\end{proof}

We are now ready to prove Theorem \ref{epsreg}.

\begin{proof}[Proof of Theorem \ref{epsreg}]
As a first step, we prove that there exists $\varepsilon_1>0$ such that if $f$ is a minimizing biharmonic map with $f \in H^2_\Lambda(B_4,N)$ and $f$ is $(\varepsilon_1,2,m)$-homogeneous, then $r_f(0) \geq 1$. Proceeding by contradiction, presume there exists a sequence of minimizing biharmonic maps $f_i$ with $\|f_i\|_{W^{2,2}(B_4)}\leq \Lambda$ and each $f_i$ is $(i^{-1},2,m)$-homogeneous, but $r_{f_i}(0)<1$. By Theorem 1.5 of \cite{Scheven}, we note that there exists $f_\infty$ such that $f_{i} \to f_\infty$ strongly in $W^{2,2}_{\mathrm{loc}}(B_4)$ and thus $f_\infty$ is $(0,2,m)$-homogeneous. Lemma \ref{lemmatechnical} gives the necessary contradiction.
We now appeal to Lemma \ref{epsreglemma} to finish the proof.

Now suppose $f$ is minimizing and $N$ satisfies Condition M. We first prove that given $\varepsilon_1>0$, there exists $\varepsilon$ sufficiently small such that if $f$ is $(\varepsilon,2,m-k_0-1)$-homogeneous then $f$ is $(\varepsilon_1,2,m)$-homogeneous. We proceed as usual by contradiction. Presume there exists a sequence of minimizing biharmonic maps $f_i$ with $\|f_i\|_{W^{2,2}(B_4)}\leq \Lambda$, each $f_i$ is $(i^{-1},2,m-k_0-1)$ homogeneous, but not $(\varepsilon_1,2,m)$-homogeneous. There exists $f \in W^{2,2}(B_4)$ such that a subsequence of the $f_i \to f$ strongly in $W^{2,2}(B_3)$ and thus $f$ is $(0,2,m-k_0-1)$ homogeneous but not $(\varepsilon_1,2,m)$ homogeneous. Lemma \ref{homogeneouslimits} implies that  $f$ is minimizing in $B_3$. Without loss of generality, presume that  $\frac{\partial  f}{\partial x_i} \equiv 0$ for all $i = \{k_0+2, \dots, m\}$. Let $\hat f:B_3\subset \Real^{k_0+1}\to N$  such that $\hat f = f|_{B_3^{k_0+1}\times \{0\}}$. Since $\hat f$ is minimizing in $B_3^{k_0+1}$, Theorem 1.1 of \cite{Scheven} implies that $\hat f:B_3^{k_0+1} \to N$ satisfies $\dim \mathcal S(\hat f) \leq k_0+1-(k_0+2)<0$. Thus, $\hat f\in C^{\infty}(B_3^{k_0+1}\backslash \{0\})$ and $\hat f$ is $0$-homogeneous. By Condition $M$, $\hat f$ must be a constant map and thus $f$ is $(0,2,m)$-homogeneous. The strong convergence of $f_i \to f$ implies a contradiction. From here, the proof follows as in the previous case.

The stationary argument is nearly identical, though now we replace the dimension reduction result for minimizers with Theorem 1.2 of \cite{Scheven} and use the compactness theory for stationary biharmonic maps, Theorem 1.6 of \cite{Scheven}. 
\end{proof}
We are now ready to prove Theorem \ref{MinMinkowski}.
\begin{proof}[Proof of Theorem \ref{MinMinkowski}]
If $x \in \mathcal B_r(f)$, then by Theorem \ref{epsreg}, $f$ is not $(\varepsilon,2r,m-4)$-homogeneous at $x$. In other words, $x \in\mathcal S^{m-5}_{\eta,2r}$ for any $0<\eta \leq \varepsilon(m,N,\Lambda)$ of Theorem \ref{epsreg}. Therefore, for every $\eta>0$ sufficiently small,
\[
T_r(\mathcal B_r(f)) \subset T_r(\mathcal S_{\eta,2r}^{m-5}).
\]Using Theorem \ref{minkowskibound} we note that
\[
\mathrm{Vol}(T_r(\mathcal B_r(f))\cap B_1(x)) \leq \mathrm{Vol}(T_r(\mathcal S_{\eta,2r}^{m-5}\cap B_1(x)) \leq C(m,N,\Lambda,\eta)r^{5-\eta}.
\]

The second item in the theorem follows the same argument, though this time we use $ \varepsilon(m,N,\Lambda)$ of the second item in Theorem \ref{epsreg}. Again, if $x \in \mathcal B_r(f)$ then $x \in \mathcal S_{\eta,2r}^{m-(k_0+2)}$ for all $0<\eta\leq \varepsilon$ so Theorem \ref{minkowskibound} again implies the result.
\end{proof}

\subsection{Uniform bounds on the number of singular points}
We prove Theorem \ref{counting} by first demonstrating that, for $m=5$ or $m=k_0+2$ for special targets, if $f$ is almost-homogeneous on a fixed scale, then $f$ is smooth on an annulus of fixed scale. Using this fact, we  demonstrate that the singular points are locally isolated. We then apply an inductive covering argument as in \cite{NVV}. The covering argument relies on the fact that the number of scales on which $f$ is not almost-homogeneous about a point is bounded, independent of $f$.

\begin{lem}\label{smooth}
Let $f:B_4(0)\subset \Real^m \to N^n$ denote a biharmonic map with $f\in H^2_\Lambda(B_{4}(0),N)$. Let $x \in B_1(0)$ and $0<r<1/2$.
\begin{itemize}
\item If $f$ is minimizing and $m=5$ then there exists $\varepsilon=\varepsilon(m,N,\Lambda)>0$ such that if $f$ is $(\varepsilon,r,0)$-homogeneous at $x$ then $f \in C^\infty(B_{\frac r2}(x) \backslash B_{\frac r4}(x))$.

\item If $f$ is minimizing and $N$ satisfies Condition $M$ or $f$ is stationary and $N$ satisfies Condition S, then for $m=k_0+2$ there exists $\varepsilon=\varepsilon(m,N,\Lambda)>0$ such that if $f$ is $(\varepsilon,r,0)$-homogeneous at $x$ then $f \in C^\infty(B_{\frac r2}(x) \backslash B_{\frac r4}(x))$.
\end{itemize}
\end{lem}
\begin{proof}Without loss of generality, let $x=0$ and $r=1$. 
We first demonstrate that given $\eta>0$ there exists $\varepsilon=\varepsilon(\eta,m,N,\Lambda)>0$ such that if $f$ is minimizing biharmonic and $(\varepsilon,1,0)$-homogeneous then there exists a $0$-homogeneous minimizer $\hat f \in C^{\infty}(\Real^5\backslash \{0\})\cap H_\Lambda^2(B_3(0),N)$ such that 
\begin{equation}\label{countingboundeq}
\fint_{B_1(0)}|f-\hat f|^2 <\eta.
\end{equation}
If not then there exists a sequence $f_i$ minimizing biharmonic and $(1/i,1,0)$-homogeneous. Since there exists an $f$ such that $f_i \to f$ strongly in $W^{2,2}_{loc}(B_{4}(0),N)$ we observe that $f$ is $0$-homogeneous a.e. and by Theorem 1.5 of \cite{Scheven}, $f$ is stationary. Thus, there exists $\tilde f\in H^2_\Lambda(B_3(0),N)$, $0$-homogeneous with $\fint_{B_1(0)} |f -\tilde f|^2 =0$. Since $\tilde f$ is thus stationary biharmonic, results of \cite{Struwe,Wang} imply that $\tilde f \in C^{\infty}(\Real^5 \backslash \{0\})$. By Lemma \ref{homogeneouslimits}, (since $f_i \to \tilde f$ in $W^{2,2}_{loc}$) we conclude that $\tilde f$ is a minimizer. Thus we have a contradiction. 

Now, let $\hat f$ satisfy \eqref{countingboundeq} and the conditions mentioned previous to the equation. 
\begin{claim}
There exists $C=C(N,\Lambda)$ such that
\[
\sup_{x \in B_{1/2}(0)\backslash B_{1/4}(0)}|\nabla \hat f(x)|\leq C.
\]
\end{claim}
\begin{proof}
Suppose no such $C$ exists. Then there exist $f_i \in H^2_\Lambda(B_3(0),N)$, $0$-homogeneous minimizers that are smooth away from the origin and $x_i \in B_{1/2}(0)\backslash B_{1/4}(0)$ with $|\nabla f_i(x_i)|> i$. Then $f_i \to f \in W^{2,2}_{loc}(B_3(0),N)$ and $f$ is a $0$-homogeneous minimizer. Choose $r_i \to 0$ and take a further subsequence such that for each $i$, $x_i \in B_{r_i}(x)$. Let $R_i \to 0$ such that $R_i/r_i \to \infty$. For every $\varepsilon>0$ there exists $I$ such that for all $i \geq I$,
\begin{align*}
r_i^{2-m} \int_{B_{r_i}(x_i)}|\nabla^2 f_i|^2 + |\nabla f_i|^4& = R_i^{2-m} \int_{B_{R_i}(x_i)}|\nabla^2 f_i|^2 + |\nabla f_i|^4\\
&\leq \varepsilon+ R_i^{2-m} \int_{B_{R_i}(x_i)}|\nabla^2 f|^2 + |\nabla f|^4\\
& \leq \varepsilon + R_i^{2-m} \int_{B_{R_i+r_i}(x)}|\nabla^2 f|^2 + |\nabla f|^4. 
\end{align*}Since $f$ is smooth at $x$, for large enough $i$, 
\[R_i^{2-m} \int_{B_{R_i+r_i}(x)}|\nabla^2 f|^2 + |\nabla f|^4 <\varepsilon_1/2.
\]Here $\varepsilon_1$ is as in the proof of Lemma \ref{lemmatechnical}. But notice that for $\varepsilon<\varepsilon_1/2$, we contradict the choice of $x_i$.
\end{proof}

 Thus for any $y \in B_{\frac 12}(0)\backslash B_{\frac 14}(0)$, there exists an $r_0$ independent of $\hat f,y$ such that 
\[
\fint_{B_{r_0}(y)}|\hat f(y)-\hat f(z)|^2 dz<\varepsilon_0/2
\]where $\varepsilon_0$ is chosen from Lemma \ref{lemmatechnical}. Thus, if we choose $\eta$ such that $\frac \eta{r_0^m} < \varepsilon_0/2$ then
\begin{align*}
\fint_{B_{r_0}(y)}|f- \hat f(y)|^2 &\leq \fint_{B_{r_0}(y)}|f-\hat f|^2+\fint_{B_{r_0}(y)}|\hat f(y)-\hat f|^2 \\
&<\frac\eta{r_0^m} + \varepsilon_0/2 < \varepsilon_0.
\end{align*}Since this argument holds for any $y\in B_{\frac 12}(0)\backslash B_{\frac 14}(0)$, Lemma \ref{lemmatechnical} implies that $f \in C^{\infty}(B_{\frac 12}(0)\backslash B_{\frac 14}(0))$. 

In the case of special targets, the arguments are nearly identical. For minimizers, we establish that $\tilde f\in C^\infty(\Real^{k_0+2}\backslash \{0\})$ by Theorem 1.1 of \cite{Scheven}. In the stationary setting, Theorem 1.6 of \cite{Scheven} gives compactness of stationary biharmonic maps. Moreover, Theorem 1.2 of \cite{Scheven} implies that $\tilde f\in C^\infty(\Real^{k_0+2}\backslash \{0\})$.
\end{proof}

\begin{cor}
In either of the settings outlined above, for each $x \in \mathcal S(f)\cap B_1(0)$ there exists $r_x >0$ such that $f \in C^{\infty}(B_{r_x}(x)\backslash\{x\})$.
\end{cor}
\begin{proof}
Choose $\varepsilon>0$ as in Lemma \ref{smooth} and let $\gamma=1/3$. Choose $\delta,q$ from Lemma \ref{QRlemma}. For $x \in \mathcal S(f) \cap B_1(0)$, the monotonicity of $\Theta_f$ implies that there exists some $r_x>0$ such that $\Theta_f(x,2r_x) - \Theta_f(x,r_x/3^q)\leq \delta$. Then $f$ is $(\varepsilon, 2r_x,0)$-homogeneous at $x$ and applying Lemma \ref{smooth}, $f \in C^\infty(B_{r_x}(x)\backslash B_{r_x/2}(x))$. Since $f$ is $(\varepsilon,r,0)$-homogeneous for all $r\leq r_x$, repeated application of Lemma \ref{smooth} implies the result.
\end{proof}

We now use a covering argument to prove Theorem \ref{counting}.
\begin{proof}[Proof of Theorem \ref{counting}] Notice first that for any fixed $f$, the number of singular points in $B_1(0)$ is finite. This follows from \eqref{Nbound} and the previous corollary. That is, for any $x$ in $B_1(0)$, the number of scales $2^{-k}$ on which $f$ is not $(\varepsilon,2^{-k},0)$-homogeneous is bounded by $Q$. 

Let $S_0:= |\mathcal S(f)\cap B_1(0)|$. We will define a sequence $\{T_i\}$ such that each $T_i \in \{0,1\}$. 

For $i=1$, consider the cover $$\{B_{1/2}(x)\}_{x \in \mathcal S(f)\cap B_1(0)}$$ and choose $x_j \in  \mathcal S(f)\cap B_1(0)$ such that $C_1:=\{B_{1/2}(x_j)\}_{j \in I_1}$ is a maximal cover of $\mathcal S(f) \cap B_1(0)$ and the balls $B_{1/4}(x_k)\cap B_{1/4}(x_j) = \emptyset$ for $k\neq j$. Note that $|I_1| \leq c(m)$.

Let $B_{1/2}(x_1)$ denote a ball in $C_1$ containing the most points of $\mathcal S(f)$ and let $S_1:=|\mathcal S(f) \cap B_{1/2}(x_1)|$. If $S_1=S_0$, let $T_1=0$. If $S_1<S_0$, let $T_1=1$. Notice that the nature of the covering implies that
\[
c(m)^{-1}S_0 \leq S_1 <S_0.
\]Also, since $S_1 \neq S_0$, there exists a $y_1 \in B_1(0)\cap \mathcal S(f) \backslash B_{1/2}(x_1)$. Note that for any $z \in B_{1/2}(x_1)$, $B_2(z) \backslash B_{1/4}(z)$ contains at least one point $x_1, y_1$.

We proceed inductively by covering $B_{2^{-i}}(x_i)\cap \mathcal S(f)$ by balls of radius $2^{-i-1}$. For the cover $C_{i+1}$ let $B_{2^{-i-1}}(x_{i+1})$ denote the ball containing the maximum number of points of $\mathcal S(f)$ in $B_{2^{-i}}(x_{i})$. Since $S_0$ is finite, there exists $i^*$ for which $S_{i^*}=1$. Let $|T|:=\sum_{i=1}^{i^*} T_i$. Then clearly
\[
S_0 \leq c(m)^{|T|}.
\]If $T_i=1$, then for all $z \in B_{2^{-i}}(x_i)$, there exists $z_i \in \mathcal S(f)$ such that $2^{-i-1}\leq |z-z_i| \leq 2^{2-i}$. Namely, $z_i=x_i$ or $z_i=y_i$. Since, $x_{i^*} \in B_{2^{-i}}(x_i)$ for all $i \leq i^*$, for each $T_i=1$ there exists $z_i \in \mathcal S(f)$ such that
\[
2^{-i-1}\leq |x_{i^*}-z_i|\leq 2^{2-i}.
\]Since \eqref{Nbound} and Lemma \ref{smooth} give a bound on the number of annuli of a fixed ratio about $x_{i^*}$ for which $f$ is not smooth, $|T| \leq C(m,\Lambda,N)$.

\end{proof}

\bibliographystyle{amsplain}
\bibliography{Biblio}
\end{document}